\documentclass[11pt]{amsart}

\usepackage{amsfonts,amssymb,amsmath,amsthm,amscd,enumerate}
\usepackage{latexsym}
\usepackage{euscript}

\newtheorem{theorem}{Theorem}[section]

\newtheorem{lemma}[theorem]{Lemma}
\newtheorem{corollary}[theorem]{Corollary}

\newtheorem{conjecture}[theorem]{Conjecture}

\title{Products of commutators in a Lie nilpotent associative algebra}

\author{Galina Deryabina}

\address{Department of Computational Mathematics and Mathematical Physics (FS-11), Bauman Moscow State Technical University, 2-nd Baumanskaya Street, 5, 105005 Moscow, Russia}

\email{galina\_deryabina@mail.ru}

\author{Alexei Krasilnikov}

\address{Departamento de Matem\'atica, Universidade de Bras\'\i lia, 70910-900 Bras\'\i lia, DF, Brasil}

\email{alexei@unb.br}

\date{}

\begin{document}

\maketitle

\begin{abstract}
Let $F$ be a field and let $F \langle X \rangle$ be the free unital associative algebra over $F$ freely generated by an infinite countable set $X = \{ x_1, x_2, \dots \}$. Define a left-normed commutator $[a_1, a_2, \dots , a_n]$ recursively by $[a_1, a_2] = a_1 a_2 - a_2 a_1$, $[a_1, \dots , a_{n-1}, a_n] = [[a_1, \dots , a_{n-1}], a_n]$ $(n \ge 3)$. For $n \ge 2$, let $T^{(n)}$ be the two-sided ideal in $F \langle X \rangle$ generated by all commutators $[a_1, a_2, \dots , a_n]$ ($a_i \in F \langle X \rangle )$.

Let $F$ be a field of characteristic $0$. In 2008 Etingof, Kim and Ma conjectured that  $T^{(m)} T^{(n )} \subset T^{(m+n -1)}$ if and only if $m$ or $n$ is odd. In 2010 Bapat and Jordan confirmed the ``if'' direction of the conjecture:  if at least one of the numbers $m$, $n$ is odd then $T^{(m)} T^{(n)} \subset T^{(m + n -1)}.$ The aim of the present note is to confirm the ``only if'' direction of the conjecture. We prove that  if $m = 2 m'$ and $n = 2 n'$ are even then $T^{(m)} T^{(n)} \nsubseteq T^{(m +n -1)}.$ Our result is valid over any field $F$.
\end{abstract}



\section{Introduction}

Let $F$ be a field. Let $X = \{x_1, x_2, \dots \} $ be an infinite countable set and let $F \langle X \rangle$ be the free associative algebra over $F$ freely generated by $X$. Define a left-normed commutator $[a_1, a_2, \dots , a_n]$ recursively by $[a_1, a_2] = a_1 a_2 - a_2 a_1$, $[a_1, \dots , a_{n-1}, a_n] = [[a_1, \dots , a_{n-1}], a_n]$ $(n \ge 3)$. For $n \ge 2$, let $T^{(n)}$ be the two-sided ideal in $F \langle X \rangle$ generated by all commutators $[a_1, a_2, \dots , a_n]$ ($a_i \in F \langle X \rangle )$.

In 2008 Etingof, Kim and Ma \cite{EKM09} made a conjecture (see Conjecture 3.6 in the arXiv version of \cite{EKM09}) that can be reformulated as follows: 

\begin{conjecture}[see \cite{EKM09}]
\label{conject}
Let $F$ be a field of characteristic $0$. Then  $T^{(m)} T^{(n)} \subset T^{(m + n -1)}$ if and only if $m$ or $n$ is odd. 
\end{conjecture}

\noindent
In \cite{EKM09} this conjecture was confirmed for $m$ and $n$ such that $m + n \le 7$. In 2010 Bapat and Jordan \cite[Corollary 1.4]{BJ10} confirmed the ``if'' direction of the conjecture for arbitrary $m, n$. 

\begin{theorem}[see \cite{BJ10}]
\label{BJ}
Let $F$ be a field of characteristic $\ne 2,3$. Let $m, n \in \mathbb Z,$ $m, n >1$ and at least one of the numbers $m$, $n$ is odd. Then
\begin{equation}
\label{tmtn}
T^{(m)} T^{(n)} \subset T^{(m + n -1)}.
\end{equation}
\end{theorem}

The aim of the present note is to confirm the ``only if'' direction of the conjecture. Our main result is as follows.

\begin{theorem}
\label{maintheorem1}
Let $F$ be a field and let $m = 2m'$, $n = 2n'$ be arbitrary positive even integers. Then
\[
T^{(m)} T^{(n)} \nsubseteq T^{(m + n -1)}.
\]
\end{theorem}
 
Recall that an associative algebra $A$ is Lie nilpotent of class at most $c$ if $[u_1, \dots , u_c , u_{c+1}]=0$ for all $u_i \in A$. We deduce Theorem \ref{maintheorem1} from the following result.

\begin{theorem}
\label{maintheorem2}
Let $F$ be a field and let $m=2m'$, $n = 2n'$ be arbitrary positive even integers.  Then there exists a  unital associative algebra $A$ such that the following two conditions are satisfied:

i) for all $u_1, u_2, \dots , u_{m + n -1} \in A$ we have 
\[
[u_1, u_2, \dots , u_{m + n -1}]=0,
\]
that is, the algebra $A$ is Lie nilpotent of class at most $m+n-2$;

ii) there are $v_1,  \dots , v_{m} , w_1, \dots , w_{n} \in A$ such that 
\[
[v_1, \dots , v_m] [w_{1}, \dots , w_{n}] \ne 0.
\]
\end{theorem}

\noindent
If $F$ is a field of characteristic $\ne 2$ then in Theorem \ref{maintheorem2} one can take $A = E \otimes E_r$ where  $E$ is the infinite-dimensional unital Grassmann algebra and $E_r$ is the $r$-generated unital Grassmann algebra for  $r = m + n -4$.

\bigskip
\noindent
\textbf{Remarks.}
1. Note that if $k > \ell$ then $T^{(k)} \subset T^{(\ell )}$; in particular, $T^{(m+n-1)} \subset T^{(m+n-2)}.$ Let $R$ be an arbitrary associative and commutative unital ring and let $m, n \in \mathbb Z,$ $m, n >1$. Then in $R \langle X \rangle$ we have
\[
T^{(m)} T^{(n )} \subset T^{(m+n -2)}.
\]
This assertion was proved by Latyshev \cite[Lemma 1]{Latyshev65} in 1965 (Latyshev's paper was published in Russian) and independently rediscovered by Gupta and Levin \cite[Theorem 3.2]{GL83} in 1983.

2. The proof of Theorem \ref{BJ} given in \cite{BJ10} is valid for algebras over  an associative and commutative unital ring $R$ such that $\frac{1}{6} \in R$. In fact, Theorem \ref{BJ} holds over any $R$ such that $\frac{1}{3} \in R$ (see \cite[Remark 3.9]{AE15} for explanation). Moreover, for some $m$ and $n$ (\ref{tmtn}) holds over an arbitrary ring $R$: for instance, $T^{(3)} T^{(3)} \subset T^{(5)}$ in $R \langle X \rangle$ for any $R$ (see \cite[Lemma 2.1]{CostaKras13}). However, in general Theorem \ref{BJ} fails over $\mathbb Z$ and over a field of characteristic $3$: it was shown in \cite{DK15,Kr13} that in this case $T^{(3)} T^{(2)} \nsubseteq T^{(4)}$ and moreover, $T^{(3)} \bigl( T^{(2)}\bigr) ^{\ell} \nsubseteq T^{(4)}$ for all $\ell \ge 1$.

3. In 1978  Volichenko proved Theorem \ref{BJ} for $m=3$ and arbitrary $n$ in the preprint \cite{Volichenko78} written in Russian;  in 2007 Gordienko \cite{Gordienko07} independently proved this theorem for $m=3, n=2$. These results were unknown to the authors of \cite{BJ10,EKM09}. Recently another proof of Theorem \ref{BJ} has been published in \cite{GrishinPchel15}.

4. In \cite{EKM09} a pair $(m, n)$ of positive integers was called \textit{null} if for each algebra $A$ (over a field $F$ of characteristic 0) $T^{(m)} (A) \ T^{(n)} (A) \subset T^{(m+n-1)} (A)$ where $T^{( \ell )} (A)$ is the two-sided ideal in $A$ generated by all commutators $[a_1, \dots , a_{\ell}]$ $(a_i \in A)$. The original conjecture stated in \cite[Conjecture 3.6]{EKM09} was as follows: A pair $(m,n)$ is null if and only if $m$ or $n$ is odd. This conjecture is equivalent to Conjecture \ref{conject} above; this can be checked using the same argument that is used to deduce Theorem \ref{maintheorem1} from Theorem \ref{maintheorem2}.

\section{Proofs of Theorems \ref{maintheorem1} and \ref{maintheorem2} }

First we prove some auxiliary results.

Let $G$ and $H$ be unital associative algebras over a field $F$ such that $[g_1, g_2, g_3] = 0$, $[h_1, h_2, h_3] =0$ for all $g_i \in G$, $h_j \in H$. Note that each commutator $[g_1, g_2]$ $(g_i \in G)$ is central in $G$, that is, $[g_1, g_2] g = g [g_1, g_2]$ for each $g \in G$. Similarly, each commutator $[h_1,h_2]$ $(h_j \in H)$ is central in $H$.

\begin{lemma}
\label{cl}
Let
\[
c_{\ell} = [ g_1 \otimes h_1, g_2 \otimes h_2, \dots , g_{\ell} \otimes h_{\ell}] 
\]
where $\ell \ge 2, g_i \in G, h_j \in H$. Then
\begin{align*}
c_2 = & \ [g_1,g_2] \otimes h_1 h_2 + g_2 g_1 \otimes [h_1, h_2],
\\
c_{2k} = & \ [g_1,g_2] [g_3,g_4] \dots [g_{2k-1},g_{2k}] \otimes [h_1 h_2, h_3] [h_4,h_5] \dots [h_{2k-2}, h_{2k-1}] h_{2k} 
\\
+ & \ [g_2 g_1, g_3] [g_4,g_5] \dots [g_{2k-2}, g_{2k-1}] g_{2k} \otimes [h_1, h_2] [h_3,h_4] \dots [h_{2k-1}, h_{2k}] \qquad (k >1),
\\
c_{2k+1} = & \ [g_1, g_2] [g_3,g_4]  \dots [g_{2k-1}, g_{2k}] g_{2k+1} \otimes [h_1 h_2, h_3][h_4,h_5] \dots [h_{2k}, h_{2k+1}]
\\
+ & \ [g_2 g_1, g_3] [g_4, g_5] \dots [g_{2k}, g_{2k+1}] \otimes [h_1, h_2] [h_3,h_4] \dots [h_{2k-1}, h_{2k}] h_{2k+1} \qquad (k \ge 1).
\end{align*}
\end{lemma}

\begin{proof}
Induction on the length $\ell$ of the commutator $c_{\ell}$. If $\ell = 2$ then
\begin{align*}
c_2 = & \ [g_1 \otimes h_1, g_2 \otimes h_2] = g_1 g_2 \otimes h_1 h_2 - g_2 g_1 \otimes h_2 h_1
\\
= & \ g_1 g_2 \otimes h_1 h_2 - g_2 g_1 \otimes h_1 h_2 + g_2 g_1 \otimes h_1 h_2 - g_2 g_1 \otimes h_2 h_1
\\
= & \ [g_1, g_2] \otimes h_1 h_2 + g_2 g_1 \otimes [h_1, h_2].
\end{align*}

Let $\ell >2$; suppose that for each $\ell ' < \ell$ the lemma has already been proved.

Let $\ell = 2k+1$ $(k \ge 1)$. By the induction hypothesis, we have
\begin{align*}
c_{2k+1} = & \ [c_{2k}, g_{2k+1} \otimes h_{2k+1}]
\\
= & \ \bigl[ [g_1,g_2] \dots [g_{2k-1},g_{2k}]  \otimes [h_1 h_2, h_3] [h_4,h_5] \dots [h_{2k-2}, h_{2k-1}] h_{2k} , g_{2k+1} \otimes h_{2k+1} \bigr]
\\
+ & \ \bigl[ [g_2 g_1, g_3] [g_4,g_5] \dots [g_{2k-2}, g_{2k-1}] g_{2k} \otimes [h_1, h_2] \dots [h_{2k-1}, h_{2k}] , g_{2k+1} \otimes h_{2k+1} \bigr] .
\end{align*}
Note that the products $[g_1,g_2] \dots [g_{2k-1},g_{2k}]$ and $[h_1 h_2, h_3] [h_4,h_5] \dots [h_{2k-2}, h_{2k-1}]$  are central in $G$ and $H$, respectively, so
\begin{align*}
& \ \bigl[ [g_1,g_2] \dots [g_{2k-1},g_{2k}]  \otimes [h_1 h_2, h_3] [h_4,h_5] \dots [h_{2k-2}, h_{2k-1}] h_{2k} , g_{2k+1} \otimes h_{2k+1} \bigr]
\\
= & \ [g_1,g_2] \dots [g_{2k-1},g_{2k}] g_{2k+1} \otimes [h_1 h_2, h_3] [h_4,h_5] \dots [h_{2k-2}, h_{2k-1}] h_{2k} h_{2k+1}
\\
- & \ g_{2k+1}  [g_1,g_2] \dots [g_{2k-1},g_{2k}] \otimes h_{2k+1} [h_1 h_2, h_3] [h_4,h_5] \dots [h_{2k-2}, h_{2k-1}] h_{2k} 
\\
= & \ [g_1,g_2] \dots [g_{2k-1},g_{2k}] g_{2k+1} \otimes [h_1 h_2, h_3] [h_4,h_5] \dots [h_{2k-2}, h_{2k-1}] h_{2k} h_{2k+1}
\\
- & \ [g_1,g_2] \dots [g_{2k-1},g_{2k}] g_{2k+1} \otimes [h_1 h_2, h_3] [h_4,h_5] \dots [h_{2k-2}, h_{2k-1}]  h_{2k+1} h_{2k}
\\
= & \ [g_1,g_2] \dots [g_{2k-1},g_{2k}] g_{2k+1} \otimes [h_1 h_2, h_3] [h_4,h_5] \dots [h_{2k-2}, h_{2k-1}] [h_{2k}, h_{2k+1}].
\end{align*}
Similarly,
\begin{align*}
& \bigl[ [g_2 g_1, g_3] [g_4,g_5] \dots [g_{2k-2}, g_{2k-1}] g_{2k} \otimes [h_1, h_2] \dots [h_{2k-1}, h_{2k}] , g_{2k+1} \otimes h_{2k+1} \bigr] 
\\
= & \ [g_2 g_1, g_3] [g_4,g_5] \dots [g_{2k-2}, g_{2k-1}] [g_{2k}, g_{2k+1}] \otimes [h_1, h_2] \dots [h_{2k-1}, h_{2k}] h_{2k+1}
\end{align*}
so 
\begin{align*}
c_{2k+1} = & \ [g_1, g_2] \dots [g_{2k-1}, g_{2k}] g_{2k+1} \otimes [h_1 h_2, h_3][h_4,h_5] \dots [h_{2k}, h_{2k+1}]
\\
+ & \ [g_2 g_1, g_3] [g_4, g_5] \dots [g_{2k}, g_{2k+1}] \otimes [h_1, h_2] \dots [h_{2k-1}, h_{2k}] h_{2k+1} ,
\end{align*}
as required.

Let $\ell = 2k$ $(k > 1)$. By the induction hypothesis, we have
\begin{align*}
c_{2k} = & \ [c_{2k-1}, g_{2k} \otimes h_{2k}]
\\
= & \ \bigl[ [g_1, g_2] \dots [g_{2k-3}, g_{2k-2}] g_{2k-1} \otimes [h_1 h_2, h_3][h_4,h_5] \dots [h_{2k-2}, h_{2k-1}], g_{2k} \otimes h_{2k} \bigr]
\\
+ & \ \bigl[ [g_2 g_1, g_3] [g_4, g_5] \dots [g_{2k-2}, g_{2k-1}] \otimes [h_1, h_2] \dots [h_{2k-3}, h_{2k-2}] h_{2k-1} , g_{2k} \otimes h_{2k} \bigr] 
\\
= & \ [g_1, g_2] \dots [g_{2k-3}, g_{2k-2}] [g_{2k-1}, g_{2k}] \otimes [h_1 h_2, h_3][h_4,h_5] \dots [h_{2k-2}, h_{2k-1}] h_{2k}
\\
+ & \ [g_2 g_1, g_3] [g_4, g_5] \dots [g_{2k-2}, g_{2k-1}] g_{2k} \otimes [h_1, h_2] \dots [h_{2k-3}, h_{2k-2}] [h_{2k-1} , h_{2k}],
\end{align*}
as required.

This completes the proof of Lemma \ref{cl}.
\end{proof}

\begin{corollary}
\label{nilp}
Suppose that 
\begin{equation}
\label{prod}
[f_1, f_2] \dots [f_{2k-1}, f_{2k}] = 0  \qquad \mbox{for all} \ \ f_j \in H.
\end{equation}
Then for all $u_i \in G \otimes H$ we have
\[
[u_1, u_2, \dots , u_{2k+1}] = 0.
\]
\end{corollary}

\begin{proof}
Since each $u_i \in G \otimes H$ is a sum of products of the form $g \otimes h$ ($g \in G$, $h \in H$), the commutator $[u_1, u_2, \dots , u_{2k+1}]$ is a sum of commutators of the form $[ g_1 \otimes h_1, g_2 \otimes h_2, \dots , g_{2k+1} \otimes h_{2k+1}]$. On the other hand, it follows from (\ref{prod}) and  Lemma \ref{cl} that $[ g_1 \otimes h_1, g_2 \otimes h_2, \dots , g_{2k+1} \otimes h_{2k+1}] = 0$ for all $g_i \in G$, $h_j \in H$. Thus, $[u_1, u_2, \dots , u_{2k+1}] = 0$ for all $u_i \in G \otimes H$, as required.
\end{proof}

\begin{corollary}
\label{nilp2}
Let $v_1 = g_1 \otimes 1$, $v_i = g_i \otimes h_i$ $(i = 2, \dots , 2m'-1)$, $v_{2m'}=g_{2m'} \otimes 1$, $w_1 = g_1' \otimes 1$, $w_j = g_j' \otimes h_j'$ $(j = 2, \dots , 2n'-1)$, $w_{2n'}=g_{2n'}' \otimes 1$ where $g_i,g_i' \in G$, $h_j, h_j' \in H$. Then
\begin{align*}
[v_1, \dots , v_{2m'}] [w_1, \dots , w_{2n'}] =  & \ [g_1,g_2] \dots [g_{2m'-1}, g_{2m'}] [g_1', g_2'] \dots [g_{2n'-1}', g_{2n'}']   
\\
\otimes  & \ [h_2, h_3] \dots [h_{2m'-2},h_{2m'-1}] [h_2',h_3'] \dots [h_{2n'-2}',h_{2n'-1}'] .
\end{align*}
\end{corollary}

\begin{proof}
By Lemma \ref{cl}, we have
\begin{align*}
[v_1, \dots , v_{2m'}] = & \ [g_1, g_2]  \dots [g_{2m'-1}, g_{2m'}] \otimes [h_2, h_3] \dots [h_{2m'-2}, h_{2m'-1}] ,
\\
[w_1, \dots , w_{2n'}] = & \ [g_1', g_2']  \dots [g_{2n'-1}', g_{2n'}']  \otimes [h_2', h_3'] \dots [h_{2n'-2}', h_{2n'-1}'].
\end{align*}
The result follows.
\end{proof}

\begin{proof}[Proof of Theorem \ref{maintheorem2}] Two cases are to be considered: the case when $char \  F \ne 2$ and the case when $char \ F =2$.

Case 1. Suppose that $F$ is a field of characteristic $\ne 2$. Let $E$ be the unital infinite-dimensional Grassmann (or exterior) algebra over $F$. Then $E$ is generated by the elements $e_i$ $(i = 1, 2, \dots )$ such that $e_i e_j = - e_j e_i$, $e_i^2 = 0$ for all $i, j$ and the set
\[
\mathcal B = \{ e_{i_1} e_{i_2} \dots e_{i_k} \mid k \ge 0, \, i_1 < i_2 < \dots < i_k \}
\]
forms a basis of $E$ over $F$. 

It is well known that $[g_1,g_2,g_3]=0$ for all $g_i \in E$. Indeed, we may assume without loss of generality that $g_{\ell} \in \mathcal B$ $(\ell = 1,2,3)$. Let $g_{\ell} = e_{i_{\ell 1}} \dots e_{i_{\ell k(\ell )}}$ $(\ell = 1,2,3)$. Note that if $k =  2 k'$ is even then the product $e_{i_1} e_{i_2} \dots e_{i_{k}}$ is central in $E$ because it commutes with all generators $e_i$. Hence, if  $k(1)$ or $k(2)$ is even then $[g_1,g_2 ]=0$ and, therefore, $[g_1,g_2,g_3]=0$. On the other hand, if both $k(1)$ and $k(2)$ are odd then the commutator $[g_1,g_2] = 2 g_1 g_2 = 2 e_{i_{1 1}} \dots e_{i_{1 k(1 )}} e_{i_{2 1}} \dots e_{i_{2 k(2 )}}$ is central in $E$ and again $[g_1,g_2,g_3]=0$, as claimed.

Recall that the $r$-generated unital Grassmann algebra $E_r$ is the unital subalgebra of $E$ generated by $e_1, e_2, \dots , e_r$. Note that $[h_1,h_2,h_3]=0$ for all $h_j \in E_r$. 

Take $A=E \otimes E_r$ where $r = m+n-4= 2(m'+n'-2)$. We can apply Lemma \ref{cl} and Corollaries \ref{nilp} and \ref{nilp2} for $G=E$, $H=E_r$.

Let $k = m' +n' -1$. Note that $2k >r$. It follows that $[f_1,f_2] \dots [f_{2k-1},f_{2k}]=0$ for all $f_i \in E_r$. Indeed, for all $f, f' \in E_r$ the commutator $[f,f']$ belongs to the linear span of the set $\{ e_{i_1}  \dots e_{i_{2 \ell}} \mid \ell \ge 1, 1 \le i_s \le r \}$.  Hence, $[f_1,f_2] \dots [f_{2k-1},f_{2k}]$ belongs to the linear span of the set $\{ e_{i_1}  \dots e_{i_{2 \ell}} \mid \ell \ge k, 1 \le i_s \le r \}$. Since $2 \ell \ge 2k > r$, each product $e_{i_1}  \dots e_{i_{2 \ell}}$ above contains equal terms $e_{i_{s}} = e_{i_{s'}}$ $(s < s')$ and, therefore, is equal to $0$. Thus, $[f_1,f_2] \dots [f_{2k-1},f_{2k}]=0$, as claimed .

Now, by Corollary \ref{nilp}, we have $[u_1, \dots , u_{2k+1}]=0$ for all $u_i \in E \otimes E_r$ , that is,
\[
[u_1, \dots ,u_{m+n-1}] =0
\]
for all $u_1, \dots , u_{m+n-1}  \in A$, as required. 

Further, take  $v_1 = e_1 \otimes 1$, $v_i = e_i \otimes e_{i-1}$ $(i = 2, \dots , 2m'-1)$, $v_{2m'}=e_{2m'} \otimes 1$, $w_1 = e_{2m'+1} \otimes 1$, $w_j = e_{2m'+j} \otimes e_{2m'+j-3}$ $(j = 2, \dots , 2n'-1)$, $w_{2n'}=e_{2m'+2n'} \otimes 1$. Note that if $i \ne j$ then $[e_i,e_j] = 2e_i e_j$. By Corollary \ref{nilp2}, we have
\begin{align*}
[v_1, \dots , v_{2m'}] [w_1, \dots , w_{2n'}] =  & \ [e_1,e_2] \dots [e_{2m'-1}, e_{2m'}] [e_{2m'+1}, e_{2m'+2}] \dots [e_{2m'+2n'-1}, e_{2m'+2n'}]   
\\
\otimes  & \ [e_1, e_2] \dots [e_{2m'-3},e_{2m'-2}] [e_{2m'-1},e_{2m'}] \dots [e_{2m'+2n'-5},e_{2m'+2n'-4}] 
\\
=  & \ 2^{m'+n'} e_1 e_2 \dots e_{2m'-1} e_{2m'} e_{2m'+1} e_{2m'+2} \dots e_{2m'+2n'-1} e_{2m'+2n'}   
\\
\otimes  & \ 2^{m'+n'-2} e_1 e_2 \dots e_{2m'-3} e_{2m'-2} e_{2m'-1} e_{2m'} \dots e_{2m'+2n'-5} e_{2m'+2n'-4} 
\\
= & \ 2^{m+n-2} e_1 e_2 \dots e_{m+n} \otimes e_1 e_2 \dots e_{m+n-4} \ne 0,
\end{align*}
as required.


Case 2. Suppose that $F$ is a field of characteristic $2$. Let $\mathcal G$ be the group given by the presentation
\[
\mathcal G = \langle y_1, y_2, \dots \mid y_i^2 =1, \ \big( (y_i, y_j), y_k \big) = 1 \ (i,j,k = 1,2, \dots ) \rangle
\]
where $(a,b)=a^{-1}b^{-1} a b$. Then $\mathcal G$ is a nilpotent group of class $2$ so $(a,b) c = c(a,b)$ and $(a,bc) = (a,b) (a,c)$ for all $a,b,c \in \mathcal G$. The quotient group $\mathcal G / \mathcal G'$ is an elementary abelian $2$-group so $a^2 \in \mathcal G'$ for all $a \in \mathcal G$. Hence, $(a,b)^2 = (a^2,b) =1$ and $(a,b) = (a,b)^{-1} = (b,a)$ for all $a,b \in \mathcal G$.  

Let $(<)$ be an arbitrary linear order on the set $\{ (i,j) \mid i,j \in \mathbb Z, \ 0< i <j \}$. The following lemma is well known and easy to check.

\begin{lemma}
\label{groupG}
Let $a \in \mathcal G$. Then $a$ can be written in a unique way  in the form
\begin{gather}
\label{formG}
a = y_{i_1} \dots y_{i_q} (y_{j_1}, y_{j_2}) \dots (y_{j_{2q'-1}}, y_{j_{2q'}}) 
\\
 \mbox{ where }\ q,q' \ge 0;  \quad i_1< \dots < i_q, \quad j_{2s-1}< j_{2s} \mbox{ for all }  s, \quad (j_{2s-1}, j_{2s}) <  (j_{2s'-1}, j_{2s'}) \mbox{ if } \ s < s' . \nonumber
\end{gather}
\end{lemma}

Let $F \mathcal G$ be the group algebra of $\mathcal G$ over $F$. Let $d_{ij} = (y_i,y_j) + 1 \in F \mathcal G$. Note that $d_{ij} = d_{ji}$ and $d_{ii} = 0$ for all $i,j$.

Let $I$ be the two-sided ideal of $F \mathcal G$ generated by the set
\[
S = \{ d_{i_1 i_2}  d_{i_3 i_4} + d_{i_1 i_3} d_{i_2 i_4} \mid i_1, i_2, i_3, i_4 = 1,2 \dots \} .
\]
Note that $d_{j_1 j_3} d_{j_2 j_3}  \in I$ for all $j_1,j_2,j_3$ because $d_{j_1 j_3} d_{j_2 j_3} = d_{j_1 j_3} d_{j_2 j_3} + d_{j_1 j_2} d_{j_3 j_3}  \in S$.  Since $d_{i j} = d_{j i}$ for all $i,j$, we have $ d_{i_1 i_2} d_{i_3 i_4}  \in I$ if any two of the indices $i_1, i_2, i_3, i_4$ coincide. It follows that
\begin{equation}
\label{prodsum}
\prod_s (y_{j}, y_{i_s}) +1 = \prod_s (d_{j i_s} + 1)  +1  = \Big( \prod_s d_{j i_s} + \dots + \sum_{s<s'} d_{j i_s} d_{j i_{s'}}+ \sum_s d_{j i_s} + 1 \Big) + 1 \equiv \sum_s d_{j i_s}  \pmod{I}.
\end{equation}

The following two lemmas are well known (see, for instance, \cite[Lemma 2.1]{GK95}, \cite[Example 3.8]{GL83}).

\begin{lemma}
\label{nilp22}
For all $u_1, u_2, u_3 \in F \mathcal G$, we have $[u_1, u_2, u_3] \in I$. 
\end{lemma}

\begin{proof}
Let $c = \prod_s y_{i_s} \in \mathcal G$. Using (\ref{prodsum}), we have 
\begin{align*}
& \bigl( (y_{j_1}, c) +1 \bigr) \bigl( (y_{j_2}, c) + 1\bigr) =  \bigl( \prod_s (y_{j_1}, y_{i_s}) +1 \bigr) \bigl( \prod_s (y_{j_2}, y_{i_s}) + 1\bigr) 
\\
\equiv & \Big( \sum_s  d_{j_1 i_s} \Big) \Big( \sum_s d_{j_2 i_s} \Big) \pmod{I}
=  \sum_s  d_{j_1 i_s} d_{j_2 i_s} +   \sum_{s<s'} ( d_{j_1 i_s} d_{j_2 i_{s'}} +   d_{j_1 i_{s'}} d_{j_2 i_{s}} ) \equiv 0 \pmod{I} ,
\end{align*}
that is, $ \bigl( (y_{j_1}, c) +1 \bigr) \bigl( (y_{j_2}, c) + 1\bigr) \in I$ for all $c \in \mathcal G$ and all $j_1,j_2$. Similar to (\ref{prodsum}), one can check that 
\begin{equation}
\label{prodsum2}
 \prod_s (y_{i_s},c) + 1 \equiv \sum_s \bigl( (y_{i_s},c) + 1 \bigr) \pmod{I} .
\end{equation}

Let $a,b \in \mathcal G$, $a = \prod_s y_{i_s},$ $b = \prod_{s'} y_{i'_{s'}}.$ Using (\ref{prodsum2}), we have 
\begin{align*}
& \bigl( (a,c) + 1 \bigr) \bigl( (b,c) + 1 \bigr) =  \bigl( \prod_s (y_{i_s},c) + 1 \bigr) \bigl( \prod_{s'} (y_{i'_{s'}},c) + 1 \bigr)
\\
\equiv & \Big(  \sum_s \bigl( (y_{i_s},c) + 1 \bigr) \Big) \Big( \sum_{s'} \bigl( (y_{i'_{s'}},c) + 1 \bigr) \Big) \pmod{I} = \sum_{s,s'}\bigl( (y_{i_s},c) + 1 \bigr) \bigl( (y_{i'_{s'}},c) + 1 \bigr) \equiv 0 \pmod{I},
\end{align*}
that is,
\begin{equation}
\label{acbc}
\bigl( (a,c) + 1 \bigr) \bigl( (b,c) + 1 \bigr) \in I \quad \mbox{ for all } a,b,c \in \mathcal G. 
\end{equation}

Now we are in a position to complete the proof of Lemma \ref{nilp22}. It is clear that it suffices to prove that $[a,b,c] \in I$ for all $a,b,c \in \mathcal G$. Note that, for $a,b \in \mathcal G$, 
\[
[a,b] = ab (1 + b^{-1}a^{-1} ba) = ab \bigl( 1+ (b,a) \bigr) 
\]
(recall that $char \ F = 2$). We have 
\[
[a,b,c] =  \bigl[ ab \bigl( 1 + (b, a ) \bigr) ,c  \bigr] = [ab, c]  \bigl( 1 + (b, a ) \bigr)
= abc  \bigl( 1 + (c, ab) \bigr) \bigl( 1 + (b,a) \bigr) = abc  \bigl( 1 + (c, ab) \bigr) \bigl( 1 + (b,ab) \bigr) 
\]
because $(b,ab) = (b,a) (b,b) = (b,a)$. By (\ref{acbc}), we have $\bigl( 1 + (c, ab) \bigr) \bigl( 1 + (b,ab) \bigr) \in I$ and therefore $[a,b,c] \in I$, as required.
\end{proof}

\begin{lemma}
\label{notin}
For all $\ell  >0$, we have $\bigl( ({y}_{1}, {y}_{2}) + 1 \bigr) \dots \bigl(  ({y}_{{2 \ell -1}}, {y}_{{2\ell}}) + 1 \bigr)  \notin I .$
\end{lemma}

\begin{proof}
Let $\mathcal G'$ be the derived subgroup of $\mathcal G$; let $c_{ij} = (y_i, y_j)$. Then each element of $\mathcal G'$ can be written in a unique way in the form $c_{j_1 j_2} \dots c_{j_{2q-1}j_{2q}}$ where $q \ge 0$, $j_{2s-1}< j_{2s}$ for all  $s$,  $(j_{2s-1}, j_{2s}) <  (j_{2s'-1}, j_{2s'})$ if $s < s'$. 

Let $F \mathcal G'$ be the group algebra of $\mathcal G'$ over $F$, $F \mathcal G' \subset F \mathcal G$. Recall that  $d_{ij} = c_{ij} +1$. Since the set 
\[
\mathcal G' = \{ c_{j_1 j_2} \dots c_{j_{2q-1}j_{2q}} \mid q \ge 0; \ j_{2s-1}< j_{2s} \mbox{  for all }  s; \  (j_{2s-1}, j_{2s}) <  (j_{2s'-1}, j_{2s'}) \mbox{ if } s < s' \}
\] 
is a basis of $F \mathcal G'$ over $F$, so is the set
\[
\{ d_{j_1 j_2} \dots d_{j_{2q-1}j_{2q}} \mid q \ge 0; \ j_{2s-1}< j_{2s} \mbox{  for all }  s; \  (j_{2s-1}, j_{2s}) <  (j_{2s'-1}, j_{2s'}) \mbox{ if } s < s' \} .
\] 
It follows that $F \mathcal G'$ is a unital $F$-algebra generated by pairwise commuting elements $d_{ij}$ subject to the relations $d_{ij}^2 = 0$, $d_{ij} = d_{ji}$ for all $i,j$ and $d_{ii} =0$ for all $i$.

By Lemma \ref{groupG}, the group $\mathcal G$ is a disjoint union of the sets $y_{i_1} \dots y_{i_q} \mathcal G'$ $(q \ge 0, 0< i_1 < i_2 < \dots < i_q)$. Hence, $F \mathcal G$ is a direct sum of the vector subspaces $y_{i_1} \dots y_{i_q} F \mathcal G'$,
\[
F \mathcal G = \bigoplus_{q \ge 0, \ 0< i_1 < i_2 < \dots < i_q} y_{i_1} \dots y_{i_q} F \mathcal G' .
\]
Recall that $I$ is a two-side ideal of $F \mathcal G$ generated by $S$. Since $S$ is central in $F \mathcal G$, we have
\[
I = F \mathcal G \cdot S = \bigoplus_{q \ge 0, \ 0< i_1 < i_2 < \dots < i_q} y_{i_1} \dots y_{i_q} F \mathcal G' \cdot S.
\]
It follows that $I \cap F \mathcal G' = F \mathcal G' \cdot S$ so to prove the lemma one has to check that $d_{1 2} \dots d_{(2 \ell -1) {2\ell}} \notin F \mathcal G' \cdot S$, that is, to check that the product $d_{1 2} \dots d_{(2 \ell -1) {2 \ell}}$ does not belong to the ideal of $F \mathcal G'$ generated by $S$. However, this is the case because the set $S$ consists of the elements $d_{i_1 i_2} d_{i_3 i_4} + d_{i_1 i_3} d_{i_2 i_4}$. 

Indeed, let $P=F [t_i \mid i = 1,2, \dots ]$ be the $F$-algebra of (commutative) polynomials in $t_i$ and let $\mathcal I$ be the ideal of  $P$ generated by the set $\{ t_i^2 \mid i = 1,2, \dots \}$. Then the map 
$\psi (d_{ij} )\rightarrow t_i t_j + \mathcal I$ can be extended up to a homomorphism $F \mathcal G' \rightarrow P/ \mathcal I$ because $\psi (d_{ij}^2) \equiv 0 \pmod{ \mathcal I}$, $\psi (d_{ij}) = \psi (d_{ji})$ and $\psi (d_{ii} ) \equiv 0 \pmod{\mathcal I}$. Since $\psi (d_{i_1 i_2} d_{i_3 i_4} + d_{i_1 i_3} d_{i_2 i_4} ) = 2 t_{i_1} t_{i_2} t_{i_3} t_{i_4} + \mathcal I = \mathcal I$ (recall that $char \ F = 2$), we have $\psi (S) = 0$. However,  $\psi (d_{12} \dots d_{(2 \ell -1) 2 \ell} ) = t_1 \dots t_{2 \ell } + \mathcal I \ne 0$ so $d_{12} \dots d_{(2 \ell -1) 2 \ell} \notin F \mathcal G' \cdot S = I \cap F \mathcal G'$ and, therefore, $d_{12} \dots d_{(2 \ell -1) 2 \ell} \notin I$, as required. 
\end{proof}

Now we are in a position to complete the proof of Theorem \ref{maintheorem2}. Let $\mathcal G_r$ be the subgroup of $\mathcal G$ generated by $y_1, \dots , y_r$; let $I_r = I \cap F \mathcal G_r$. Take $G = F \mathcal G / I$, $H = F \mathcal G_r /I_r$ where $r = m+n-4= 2(m'+n'-2)$. Take $A = G \otimes H$. By Lemma \ref{nilp22}, we can apply Lemma \ref{cl} and Corollaries \ref{nilp} and \ref{nilp2}.

Let $k = m' +n' -1$; note that $2k >r$. We claim that  $[f_1,f_2] \dots [f_{2k-1},f_{2k}] \in I_r$ for all $f_i \in F \mathcal G_r$. Indeed, we may assume without loss of generality that $f_i \in \mathcal G_r$ for all $i$. Then
\[
[f_1,f_2] \dots [f_{2k-1},f_{2k}] = f_1 f_2 \dots f_{2k} \bigl( (f_1,f_2)+1 \bigr) \dots \bigl( (f_{2k-1}, f_{2k}) + 1 \bigr) .
\]
It is clear that, for each $s$, $(f_{2s-1}, f_{2s}) = \prod_t c_{i_{st} j_{st}} $ for some commutators $c_{i_{st} j_{st}} = (y_{i_{st}}, y_{j_{st}})$. Let $d_{i_{st} j_{st}} = c_{i_{st} j_{st}} + 1$; then $c_{i_{st} j_{st}} = d_{i_{st} j_{st}} + 1$.  We have
\[
(f_{2s-1}, f_{2s}) + 1 = \prod_t c_{i_{st} j_{st}}  +1  = \Big( \prod_t ( d_{i_{st} j_{st}}   +1) \Big) +1 =  \prod_t d_{i_{st} j_{st}}   + \dots +  \sum_{t<t'} d_{i_{st} j_{st}}  d_{i_{st'} j_{st'}}  + \sum_t d_{i_{st} j_{st}} .
\]
It follows that the product $\bigl( (f_1,f_2)+1 \bigr) \dots \bigl( (f_{2k-1}, f_{2k}) + 1 \bigr)$ can be written as a sum of products of the form 
\begin{equation}
\label{product2}
d_{q_1 q_2} \dots d_{q_{2 \ell -1} q_{2 \ell}} =   \bigl( (y_{q_1}, y_{q_2}) + 1 \bigr) \dots \bigl( (y_{q_{2 \ell -1}}, y_{q_{2 \ell}}) + 1 \bigr)
\end{equation}
where $\ell \ge k$. Since $2 \ell \ge 2k >r$, in the product ( \ref{product2}) we have $q_{t}= q_{t'}$ for some $t < t'$. It follows that each product (\ref{product2}) belongs to $I_r$ and so does the product $\bigl( (f_1,f_2)+1 \bigr) \dots \bigl( (f_{2k-1}, f_{2k}) + 1 \bigr)$. Hence,  $[f_1,f_2] \dots [f_{2k-1},f_{2k}] \in I_r$, as claimed.

For any $u \in F \mathcal G$, let $\bar{u} = u + I \in F \mathcal G /I$. Since one can view the algebra $F \mathcal G_r /I_r$ as a subalgebra of $F \mathcal G /I$, we also write $\bar{u} = u + I_r \in F \mathcal G_r /I_r$ for $u \in F \mathcal G_r$. 

By the claim above, 
$[\bar{f}_1,\bar{f}_2] \dots [\bar{f}_{2k-1},\bar{f}_{2k}]=0$ for all $\bar{f}_i \in H$. Hence, by Corollary \ref{nilp}, we have $[u_1, \dots , u_{2k+1}]=0$ for all $u_i \in G \otimes H$ , that is,
\[
[u_1, \dots ,u_{m+n-1}] =0
\]
for all $u_1, \dots , u_{m+n-1}  \in A$, as required. 

Further, take  $v_1 = \bar{y}_1 \otimes 1$, $v_i = \bar{y}_i \otimes \bar{y}_{i-1}$ $(i = 2, \dots , 2m'-1)$, $v_{2m'}=\bar{y}_{2m'} \otimes 1$, $w_1 = \bar{y}_{2m'+1} \otimes 1$, $w_j = \bar{y}_{2m'+j} \otimes \bar{y}_{2m'+j-3}$ $(j = 2, \dots , 2n'-1)$, $w_{2n'}=\bar{y}_{2m'+2n'} \otimes 1$. Note that $[\bar{y}_i,\bar{y}_j] = \bar{y}_i \bar{y}_j \bigl( (\bar{y}_j, \bar{y}_i) + 1 \bigr) = \bar{y}_i \bar{y}_j \bigl( (\bar{y}_i, \bar{y}_j) + 1 \bigr) $. By Corollary \ref{nilp2}, we have
\begin{align*}
[v_1, \dots , v_{2m'}] [w_1, \dots , w_{2n'}] =  & \ [\bar{y}_1,\bar{y}_2] \dots [\bar{y}_{2m'-1}, \bar{y}_{2m'}] [\bar{y}_{2m'+1}, \bar{y}_{2m'+2}] \dots [\bar{y}_{2m'+2n'-1}, \bar{y}_{2m'+2n'}]   
\\
\otimes  & \ [\bar{y}_1, \bar{y}_2] \dots [\bar{y}_{2m'-3},\bar{y}_{2m'-2}] [\bar{y}_{2m'-1},\bar{y}_{2m'}] \dots [\bar{y}_{2m'+2n'-5},\bar{y}_{2m'+2n'-4}] 
\\
=  & \  \bar{y}_1 \bar{y}_2 \dots \bar{y}_{2m'+2n'}  \bigl( (\bar{y}_1, \bar{y}_2) + 1 \bigr)  \dots \bigl( (\bar{y}_{2m'+2n'-1}, \bar{y}_{2m'+2n'}) + 1 \bigr) 
\\
\otimes  & \ \bar{y}_1 \bar{y}_2 \dots  \bar{y}_{2m'+2n'-4}  \bigl( (\bar{y}_1, \bar{y}_2) + 1 \bigr)  \dots \bigl( (\bar{y}_{2m'+2n'-5}, \bar{y}_{2m'+2n'-4}) + 1 \bigr)  
\end{align*}
so, by Lemma \ref{notin}, $ [v_1, \dots , v_{2m'}] [w_1, \dots , w_{2n'}] \ne  \ 0$, as required.

This completes the proof of Theorem \ref{maintheorem2}.
\end{proof}

\begin{proof}[Proof of Theorem \ref{maintheorem1}]
Let $A$ be the algebra described in Theorem \ref{maintheorem2}. Define a homomorphism $\phi : F \langle X \rangle \rightarrow A$ by 
\[
\phi (x_i) =
\left\{
\begin{array}{ccc}
v_i  & \mbox{ if } & i = 1,\dots ,m;
\\
w_{i-m} & \mbox{ if } & i = m+1, \dots , m+n;
\\
0 & \mbox{ if } & i>m+n.
\end{array}
\right.
\]
Then, on one hand, $\phi \bigl( T^{(m+n-1)} \bigr) = 0$ by the item i) of Theorem \ref{maintheorem2}. On the other hand, 
\[
\phi \bigl( [x_1, \dots , x_m] [x_{m+1}, \dots , x_{m+n}] \bigr) = [v_1, \dots , v_m] [w_1, \dots , w_n] \ne 0
\]
by the item ii) of Theorem \ref{maintheorem2} so $\phi \bigl( T^{(m)} T^{(n)} \bigr) \ne 0$. It follows that 
\[
T^{(m)} T^{(n)} \nsubseteq T^{(m + n -1)},
\]
as required.
\end{proof}

\bigskip
\noindent
\textbf{Remarks.}
1. For each $\ell \ge 1$, one can choose elements $z_1, \dots , z_{2 \ell}$ in the algebra $A$ described in Theorem \ref{maintheorem2} in such a way that 
\[
[v_1, \dots , v_m] [w_1, \dots , w_n][z_1,z_2] \dots [z_{2 \ell - 1}, z_{2 \ell}] \ne 0
\]
in $A$. For instance, if $char \ F \ne 2$ then one can choose $z_i = e_{m+n+i} \otimes 1$ $(i = 1, \dots , 2 \ell )$. It follows that if $m = 2m'$ and $n = 2 n'$ are even positive integers then, for each $\ell \ge 1$,
\[
T^{(m)} T^{(n)} \bigl( T^{(2)} \bigr)^{\ell} \nsubseteq T^{(m+n-1)}.
\]

2. Let $X_k = \{ x_1, x_2, \dots , x_k \}$ and let $F \langle X_k \rangle$ be the free unital associative $F$-algebra freely generated by $X_k$. Let $T^{(n)}_k=T^{(n)}(F \langle X_k \rangle )$ be the two-sided ideal of $F \langle X_k \rangle$ generated by all commutators $[a_1, a_2, \dots , a_n]$ $(a_i \in F \langle X_k \rangle )$. If $k \ge m+n$ then Theorem \ref{maintheorem1} holds for the ideals $T^{(n)}_k$, with the same proof. However, Theorem \ref{maintheorem1} fails, in general, for small $k$: for instance, one can check that  if $k \le 3$ then $T^{(2)}_k T^{(2)}_k \subset T^{(3)}_k$. Moreover, Dangovski \cite[Theorem 3.1]{Dangovski15} has recently proved that $T^{(m)}_2 T^{(n)}_2 \subset T^{(m+n-1)}_2$ for all $m,n \ge 2$ so Theorem \ref{maintheorem1} always fails for $k = 2$. 

3. To prove Theorem \ref{maintheorem2} one can choose the algebra $A$ different from one used in our proof. For example, let $F$ be any field and let $r = m+n-4= 2(m'+n'-2)$. Let $A = F \langle X \rangle /T^{(3)} \otimes F \langle X_r \rangle /T^{(3)}_r$ where $X_r = \{ x_1, \dots , x_r \}$ and $T^{(3)}_r = T^{(3)} \bigl( F \langle X_r \rangle \bigr) =  T^{(3)} \cap F \langle X_r \rangle.$ Then $A$ satisfies the conditions i) and ii) of Theorem \ref{maintheorem2}; one can check this using a description of a basis of $F \langle X \rangle /T^{(3)}$ over $F$. Such a description can be deduced, for instance, from \cite[Proposition 3.2]{BEJKL12} or found (if $char \ F \ne 2$) in \cite[Proposition 9]{BKKS10}.

Our choice of the algebra $A$ in the proof of Theorem \ref{maintheorem2} was made with a purpose to have the paper self-contained. 

4. The tensor products of the form $E \otimes E_r \otimes \dots \otimes E_s$ were used to study the polynomial identities of  Lie nilpotent associative algebras over a field of characteristic $0$ by Drensky \cite[Section 5]{Drensky84}.

\section*{Acknowledgments}

This work was partially supported by CNPq grants 307328/2012-0 and 480139/2012-1 and by RFBR grant 15-01-05823.

\end{document}